\documentclass[a4paper,10pt]{amsart}
 \usepackage{amsfonts,mathrsfs}
 \usepackage{amsthm}
 \usepackage{amssymb}
 \usepackage{enumerate}
 \usepackage{tikz-cd}
\theoremstyle{definition}
\newtheorem{Def}{Definition}[section]
\newtheorem{example}[Def]{Example}
\newtheorem{Bem}[Def]{Remark}

\theoremstyle{plain}
\newtheorem{Prop}[Def]{Proposition}
\newtheorem{Thm}[Def]{Theorem}
\newtheorem*{Thm*}{Theorem}
\newtheorem{Lem}[Def]{Lemma}
\newtheorem{Kor}[Def]{Corollary}
\newtheorem*{Kor*}{Corollary}

\newtheorem*{con*}{Conjecture}
\newtheorem*{frag*}{Question}
\newtheorem*{verm*}{Vermutung}

\usepackage{mathrsfs}

\newcommand{\Nil}{\operatorname{Nil}}

\newcommand{\Hom}{\operatorname{Hom}} 
\newcommand{\Spec}{\operatorname{Spec}}

\newcommand{\Supp}{\operatorname{Supp}}

\newcommand{\codim}{\operatorname{codim}}

\newcommand{\GL}{\operatorname{{\mathbf GL}}}

\newcommand{\Quot}{\operatorname{Quot}}

\newcommand{\tr}{\operatorname{tr}}

\newcommand{\cF}{{\mathcal F}}

\newcommand{\cH}{{\mathcal H}}

\newcommand{\cL}{{\mathcal L}}

\newcommand{\cN}{{\mathcal N}}
\newcommand{\cO}{{\mathcal O}}

\newcommand{\cV}{{\mathcal V}}

\newcommand{\fP}{{\mathfrak P}}

\newcommand{\fm}{{\mathfrak m}}
\newcommand{\fn}{{\mathfrak n}}

\newcommand{\fp}{{\mathfrak p}}

\newcommand{\A}{{\mathbb A}}

\newcommand{\C}{{\mathbb C}}
\newcommand{\R}{{\mathbb R}}
\newcommand{\pp}{\mathbb{P}}
\newcommand{\Q}{{\mathbb Q}}

\newcommand{\Z}{{\mathbb Z}}

\title{On deformations of hyperbolic varieties}
\author{Mario Kummer}
\address{University of Konstanz, Germany}
\curraddr{Technische Universit\"at Berlin, Germany}
\email{kummer@tu-berlin.de} 
\author{Eli Shamovich}
\address{Department of  Mathematics\\ 
Technion - Israel Institute of Mathematics\\
Haifa, 3200003, Israel}
\curraddr{Pure Math.\ Dept., University of Waterloo, Waterloo, ON, N2L 3G1, Canada}
\email{eshamovich@uwaterloo.ca}

\newcommand{\comment}[1]{}

\begin{document}


\begin{abstract}
 In this paper we study flat deformations of real subschemes of $\mathbb{P}^n$, hyperbolic with respect to a fixed linear subspace, i.e. admitting a finite surjective and real fibered linear projection. We show that the subset of the corresponding Hilbert scheme consisting of such subschemes is closed and connected in the classical topology. Every smooth variety in this set lies in the interior of this set. Furthermore, we provide sufficient conditions for a hyperbolic subscheme to admit a flat deformation to a smooth hyperbolic subscheme. This leads to new examples of smooth hyperbolic varieties.
\end{abstract}
\maketitle

\section{Introduction}

The study of hyperbolic varieties grew out of the study of homogeneous hyperbolic polynomials. Hyperbolic polynomials were brought to light due to their applications to partial differential equations and in particular to the question whether a Cauchy problem is well posed for a given PDE (see \cite[Ch.\ 23]{HorIII} for more details). A homogeneous polynomial $f \in \R[x_0,\ldots,x_n]$ of degree $m$ is said to be hyperbolic with respect to a point $e \in \R^{n+1}$ if $f(e) \neq 0$ and if for every $x \in \R^{n+1}$ the univariate polynomial $f(x + t e)$ has precisely $m$ real roots counting multiplicities. If for every $x \in \R^{n+1}$ the polynomial $f(x + t e)$ has $m$ simple real roots, then $f$ is called strictly hyperbolic with respect to $e$. The remarkable convexity properties of hyperbolic polynomials were discovered by G\r{a}rding in \cite{gar}. G\r{a}rding proved that the connected component $C$ of $e$ in $\R^{n+1} \setminus \{ f(x) = 0 \}$ is a convex cone. If $f$ is irreducible, the closure of $C$ is called the hyperbolicity cone of $f$. Hyperbolic polynomials and their hyperbolicity cones quickly found applications outside the scope of PDEs, for example in convex analysis \cite{Bauschke}, in optimization \cite{Bauschke,Gu97,Re06} and in functional analysis \cite{intI, intII, Bra14}. See also \cite{Vppf} for additional examples of applications and references.

In 1968 Nuij considered the space of homogeneous polynomials of degree $m$ in $n+1$ variables, hyperbolic with respect to a fixed $e \in \R^{n+1}$, with the topology induced by the norm of the vector of coefficients. He proved that this space has non-empty interior (the strictly hyperbolic polynomials) and it is connected and simply connected. This is the main theorem of \cite{nuij}. In particular, every hyperbolic polynomial is a limit of strict hyperbolic polynomials and thus can be deformed to a smooth polynomial hyperbolic with respect to the same point. His result was used by Helton and Vinnikov in \cite{HV07} to prove that every smooth hyperbolic hypersurface of degree $m$ in $\pp^n(\R)$ is isotopic to a union of concentric spheres if $m$ is even and a union of concentric spheres and a single hyperplane if $m$ is odd. Furthermore, the isotopy passes only through smooth real hypersurfaces hyperbolic with respect to $e$.

From now on all of the schemes, varieties, and morphisms are assumed to be defined over $\R$. In particular $\pp^n$ stands for $\pp^n_\R$. Hyperbolic varieties were introduced by Vinnikov and the second author in \cite{Sha14}. A subvariety $X \subset \pp^n$ of dimension $k$ is said to be hyperbolic with respect to a linear subspace $E \subset \pp^n$ of dimension $n - k -1$ if $E \cap X = \emptyset$ and for every linear subspace $U \subset \pp^n$ of dimension $n-k$ that contains $E$, we have that $U$ intersects $X$ only at real points. This notion generalizes readily to subschemes of $\pp^n$. The notion of general hyperbolicity was further studied and expanded by the authors in \cite{us}. In particular, it was proved that each connected component of the real points of a smooth irreducible hyperbolic variety of dimension $k$ is homeomorphic to either $S^k$ or $\pp^k(\R)$. This provided a partial generalization of the result of Helton and Vinnikov. We show, however, in Examples \ref{ex:different_writhe} and \ref{ex:shastri} that a straightforward generalization of the Helton-Vinnikov result fails.

In order to attempt a generalization of Nuij's result to the general hyperbolic case, we need a slightly different point of view. Consider the space that parametrizes hypersurfaces of degree $m$ in $\pp^n$, that is $\pp^{\binom{m+n}{n}-1}$ and inside this consider the subset of all hypersurfaces hyperbolic with respect to $e$. Then Nuij's theorem implies that this set is connected, has non-empty interior and every hyperbolic hypersurface is in the closure (all with respect to the classical topology) of the smooth ones hyperbolic with respect to the same point. In this paper, we study the generalization of this question to the case of general hyperbolic varieties. Namely, let $P$ be a univariate polynomial over $\Q$ of degree $k$ and let $H$ be the Hilbert scheme of closed subschemes of $\pp^n$ with Hilbert polynomial $P$. This is a projective scheme over $\R$ thus on the set $H(\R)$ we have the classical topology.
In Section \ref{sec:closed} we show that the set of all closed equidimensional subschemes, that are hyperbolic with respect to a fixed linear subspace $E$ of dimension $n-k-1$, is closed inside the open subset of $H(\R)$ consisting of subschemes that do not intersect $E$ (Theorem \ref{thm:hyperbolic_locus_closed}). The main ingredients are the open mapping theorem and the curve selection lemma from real algebraic geometry.

In Section \ref{sec:open} we study the subset of subschemes hyperbolic with respect to $E$, such that the projection from $E$ induces on them a map to $\pp^k$ unramified at real points. These subschemes are the generalization of strictly hyperbolic polynomials. In particular, by \cite[Thm.\ 2.19]{us} for every equidimensional smooth subvariety hyperbolic with respect to $E$, the projection from $E$ is unramified at real points. We show that the set of such subschemes is open inside the set of all hyperbolic subschemes. Though unlike in the case of hypersurfaces, this subset can be empty or disconnected as is shown in Examples \ref{ex:different_writhe} and \ref{ex:nonsmoothable}. 

In Section \ref{sec:connected} we apply the methods of \cite{Har66} to show that every subscheme of $\pp^n$ hyperbolic with respect to $E$ can be deformed into a tight fan of linear subspaces hyperbolic with respect to $E$ such that every fiber of the deformation over a closed point is a subscheme hyperbolic with respect to $E$. In particular, the set of hyperbolic subschemes in $H(\R)$ is connected (Theorem \ref{thm:connected}).

Section \ref{sec:smoothing} starts with describing a class of first-order deformations satisfying a certain positivity condition. We call those deformations \textit{strict hyperbolic deformations}. Let $X \subset \pp^n$ be a subscheme hyperbolic with respect to $E$ and assume that there exists a strict hyperbolic deformation $\varphi$. Let $x \in H(\R)$ be the point corresponding to $X$ and let us assume that the Hilbert scheme is smooth at $x$. Then $\varphi$ corresponds to a tangent direction to the Hilbert scheme at $x$ and there exists a curve tangent to $\varphi$ that deforms $X$ to a smooth subscheme that is hyperbolic with respect to $E$. This method provides new examples of smooth hyperbolic varieties.

\section{Closedness of the Hyperbolic Locus} \label{sec:closed}

In the following, by a \textit{curve} we mean a variety of dimension one (in particular reduced). For a quasi-projective scheme $X$ defined over the reals, open and closed sets in $X(\R)$ are always with respect to the classical topology and otherwise with respect to the Zariski topology. In this section, we want to prove the closedness of the hyperbolic locus in the real part of the Hilbert scheme. The idea is to first prove it for one-dimensional families and then to extend it to the general case via the curve selection lemma. We start with recalling some lemmas from algebraic geometry.

\begin{Lem}\label{lem:curvefibercurve}
 Let $f: X \to Y$ be a dominant morphism of irreducible $\R$-varieties. Let $x \in X$ be a closed point such that $y=f(x) \in Y$ is a smooth point. Let $C \subset Y$ be an irreducible curve such that $y \in C$ is a smooth point of $C$. There is an irreducible curve $C' \subset X$ with $x \in C'$ and $f(C') \subset C$.
\end{Lem}

\begin{proof}
 This proof is taken from \cite[Lem. 2.3.7]{kummerdiss}. Without loss of generality we can assume that $X=\Spec B$ and $Y=\Spec A$ are affine schemes
 where $A\subset B$ is an extension ($f$ is dominant) of finitely generated $\R$-algebras without zero divisors.
 Let $\fm \subset A$ and $\fn \subset B$ be the maximal ideals corresponding to $y$ and $x$.
 Let $\fp \subset A$ be the prime ideal corresponding to $C$.
 Let $d=\dim A$. 
 The ideal $\fp A_{\fm}$ of the regular local ring $A_{\fm}$ is generated
 by $d-1$ elements $a_1, \ldots, a_{d-1} \in A_{\fm}$ because $A_{\fm}/ \fp A_{\fm}$ is regular of dimension one
 (cf. \cite[\S 17.F, Thm. 36]{mats}).
 The ideal $I$ of $B_{\fn}$ that is generated by $a_1,\ldots,a_{d-1}$ is contained in a prime ideal $\fP$
 of height at most $d-1$ by Krull's Height Theorem (\cite[\S 12.I, Thm. 18]{mats}). 
 Since $\dim B_{\fn} \geq d$, there is a prime ideal $\fP'$ of $B_{\fn}$ that contains $\fP$
 such that $\dim B_{\fn} /\fP'=1$. This gives the desired curve $C'$.
\end{proof}

\begin{Lem}\label{lem:curvefibercurve2}
 Let $f: X \to Y$ be a quasi-finite morphism of $\R$-varieties. Assume that every irreducible component of $X$ dominates $Y$ and let $Y$ be smooth and irreducible. Let $C \subset Y$ be a smooth irreducible curve. Then every irreducible component of $f^{-1}(C)$ dominates $C$.
\end{Lem}

\begin{proof}
  Since $f$ is quasi-finite, it suffices to show that every irreducible component of $f^{-1}(C)$ is a curve. Let $x \in f^{-1}(C)$ be a closed point. By assumption there is an irreducible component $X_i$ of $X$ with $x \in X_i$ such that $X_i\to Y$ is dominant. By Lemma \ref{lem:curvefibercurve} there is an irreducible curve $C' \subset X_i$ with $x \in C' \subset f^{-1}(C)$.
\end{proof}

\begin{Lem}\label{lem:irreddim}
 Let $Y$ be an irreducible curve and $f: X\to Y$ a flat morphism such that all fibers over closed points are of pure dimension $k$. Then every irreducible component of $X$ has dimension $k+1$.
\end{Lem}

\begin{proof}
 Let $x\in X$ be a closed point and $y=f(x)$. Then the local dimensions satisfy $$\dim_x X=\dim_y Y+\dim_x f^{-1}(y)=1+k$$by \cite[Cor. 6.1.2]{EGAIV2} since $f$ is flat.
\end{proof}

%

The following lemma is the first closedness result which easily follows from the open mapping theorem.

\begin{Lem}\label{lem:rfclosed2}
 Let $f: X \to Y$ be a morphism of curves. Let $Y$ be smooth and irreducible. Assume that every irreducible component of $X$ dominates $Y$. The set $S=\{y \in Y(\R):\, f^{-1}(\{y\}) \subset X(\R)\}$ is a closed subset of $Y(\R)$ (with respect to the Euclidean topology).
\end{Lem}

\begin{proof}
%
Since $S=Y(\R)\setminus (f(X(\C)\setminus X(\R)))$ and $X(\C)\setminus X(\R)$ is open, it suffices to show that $f: X(\C)\to Y(\C)$ is an open map (with respect to the Euclidean topology). By the open mapping theorem this is the case for smooth $X$. For the singular case let $\pi:\tilde{X}\to Y$ be the normalization map. Then $f\circ\pi$ is an open map. If $U\subset X(\C)$ is open, then $f(U)=f(\pi(\pi^{-1}(U)))=(f\circ\pi)(\pi^{-1}(U))$ is open.
\end{proof}


%
%
Now we are able to show that inside one dimensional families of projective schemes the set of hyperbolic ones is closed.

\begin{Lem}\label{lem:defalongcurve}
 Let $E \subset \pp^n$ be a linear subspace of dimension $n-k-1$. Let $C$ be a smooth irreducible curve and let $W \subset \pp^n \times C$ be a closed subscheme, flat over $C$, all of whose fibers are closed subschemes of $\pp^n$ of pure dimension $k$ that do not intersect $E$. The set $S$ of points in $C(\R)$ whose fiber is hyperbolic with respect to $E$ is closed in $C(\R)$.
\end{Lem}

\begin{proof}
 Let $\pi: \pp^n\smallsetminus E \to \pp^k$ be the linear projection from center $E$. The induced map $f=(\pi \times \textnormal{id})|_W: W \to \pp^k \times C$ is quasi-finite and dominant because no fiber of $f$ intersects $E$. Moreover, every irreducible component of $W$ dominates $\pp^k \times C$ by Lemma \ref{lem:irreddim}. For all $y \in \pp^k(\R)$ let $S_y=\{z\in C(\R):\, f^{-1}(\{(y,z)\})\subset W(\R)\}$. Every irreducible component of $f^{-1}(\{y\}\times C)$ dominates $\{y\}\times C$ by Lemma \ref{lem:curvefibercurve2}. Thus $S_y$ is closed by Lemma \ref{lem:rfclosed2}. Since $S=\cap_{y\in \pp^k(\R)} S_y$, the claim follows.
%
\end{proof}

\begin{example}
 Lemma \ref{lem:defalongcurve} becomes false if we drop the assumption on the fibers being equidimensional. Indeed, let $e=(0:0:1)\in\pp^2$ and $C=\A^1$. Consider the closed subscheme $W$ of $\pp^2 \times \A^1$ that is cut out by $x_1x_2-\epsilon x_2^2$ and $x_0^2 x_2+x_2^3$ where $\epsilon$ is the parameter of $\A^1$. Every fiber is the disjoint union of the line $x_2=0$ with the pair $P=\{(1:\epsilon\cdot \textrm{i}:\textrm{i}),(1:-\epsilon\cdot \textrm{i}:-\textrm{i})\}$ of complex conjugate points. In particular, no fiber contains $e$ and $W$ is flat over $\A^1$ since each fiber has the same Hilbert polynomial $T+3$. Moreover, the (real) line spanned by the two points from $P$ contains $e$ if and only if $\epsilon=0$. Since the other component is a real line, the set of points in $\A^1(\R)$ whose fiber is hyperbolic with respect to $e$ is $S=\{\epsilon\in\A^1:\,\epsilon\neq0\}$ and therefore not closed.
\end{example}

The next lemma shows that indeed it suffices to look at one-dimensional families. This is done via the curve selection lemma from semi-algebraic geometry.

\begin{Lem}\label{lem:curvesel2}
 Let $X$ be a quasi-projective variety over $\R$. Let $S \subset X(\R)$ be a semialgebraic subset.
 \begin{enumerate}[(i)]
  \item If $C(\R)\cap S$ is closed in $C(\R)$ for every irreducible curve $C\subset X$, then $S$ is closed in $X(\R)$.
  \item If $C(\R)\cap S$ is open in $C(\R)$ for every irreducible curve $C\subset X$, then $S$ is open in $X(\R)$.
 \end{enumerate}
\end{Lem}

\begin{proof}
 First we note that $(ii)$ follows directly from $(i)$. We prove the contrapositive of $(i)$: Let $x \in \overline{S} \smallsetminus S$. By the curve selection lemma \cite[Thm. 2.5.5]{BCR} there is a continuous semialgebraic map $f: [0,1] \to X(\R)$ such that $f(0)=x$ and $f(]0,1])\subset S$. Let $C$ be the Zariski closure of $f(]0,1])$. This is a curve with $x \in \overline{C(\R)\cap S}$ (the dimension does not increase by \cite[Thm. 2.8.8]{BCR}). Thus $C(\R)\cap S$ is not closed in $C(\R)$. There is also an irreducible component $C_0$ of $C$ such that $C_0(\R)\cap S$ is not closed.
\end{proof}

Putting everything together we get the following theorem.

\begin{Thm} \label{thm:hyperbolic_locus_closed}
 Let $E \subset \pp^n$ be a linear subspace of dimension $n-k-1$. Let $H$ be the Hilbert scheme of closed subschemes of $\pp^n$ with a given Hilbert polynomial of degree $k$. Let $H' \subset H$ be the subset consisting of all points corresponding to subschemes of pure dimension $k$ that do not intersect $E$. The subset $S$ of $H'(\R)$ corresponding to subschemes which are hyperbolic with respect to $E$ is closed with respect to the classical topology.
\end{Thm}

\begin{proof}
 Let $C \subset H'$ be an irreducible curve. By Lemma \ref{lem:curvesel2} it suffices to show that $C(\R)\cap S$ is closed in $C(\R)$. Let $\pi:\widetilde{C}\to C$ be the normalization map. Since the induced map $\widetilde{C}(\R)\to C(\R)$ is surjective and closed (cf. \cite[Prop. 4.2 and 4.3]{And96}), it suffices to show that $\pi^{-1}(C(\R)\cap S)$ is closed in $\widetilde{C}(\R)$. But since $\widetilde{C}$ is smooth, this follows from the universal property of the Hilbert scheme and Lemma \ref{lem:defalongcurve}.
\end{proof}

\begin{Bem}
 The set of points in the Hilbert scheme corresponding to equidimensional schemes is closed whereas the set of points in the Hilbert scheme corresponding to schemes that do not intersect a certain linear subspace is open. Thus the set $H'$ considered in Theorem \ref{thm:hyperbolic_locus_closed} is locally closed.
\end{Bem}

\section{Strict Hyperbolic Varieties} \label{sec:open}

In this section, we want to show that smooth hyperbolic varieties lie in the interior of the set of all hyperbolic varieties. Whereas the rough idea is the same as in the previous section, some subtleties arise that do not enable us to prove the results of both sections in a unified way.

\begin{Def}
 Let $f: X \to Y$ be a morphism. The \textit{branch locus} of $f$ is the set of all points $y \in Y$ such that the fiber $X_y=X\times_Y \kappa(y)$ is not reduced. This is in fact $f(\Supp\Omega_{X/Y})$, since the ramification locus is precisely the support of the relative differentials.
\end{Def}

Recall for example from \cite[\S20, Thm. 9]{Lor08} or \cite[Thm. 2.1]{peder} that if $A$ is a finite $\R$-algebra, then the bilinear form $A\times A\to \R,\,(x,y)\mapsto\tr_{A/\R}(xy)$ is positive semidefinite if and only if $\Spec A$ consists only of $\R$-points and it is nondegenerate if and only if $A$ is reduced, the latter being equivalent to the morphism $\Spec A\to\Spec\R$ being unramified. 

\begin{Lem}\label{lem:rfopenflat}
 Let $f: X \to Y$ be a finite flat surjective morphism of irreducible varieties. The set \[S=\{y \in Y(\R):\, 
    f^{-1}(y) \subset X(\R)\}\smallsetminus f(\Supp \Omega_{X/Y})\}\]is an open subset of $Y(\R)$.
\end{Lem}

\begin{proof}
 We can restrict to the case where $Y=\Spec A$, $X=\Spec B$ and $B$ is free as $A$-module. But then $S$ is just the set of points where the trace bilinear form $B\times B\to A,\, (a,b)\mapsto \tr_{B/A}(ab)$ is positive definite. This is an open set.
\end{proof}

%
We want to decude the statement of the previous lemma for not necessarily flat morphisms $X\to Y$ which we will need later. 

\begin{Lem}\label{lem:rfopencurves}
 Let $f: X \to Y$ be a finite morphism of schemes where $Y$ is an irreducible curve. The set \[S=\{y \in Y(\R):\, 
    f^{-1}(y)\subset X(\R)\}\smallsetminus f(\Supp \Omega_{X/Y})\}\]is an open subset of $Y(\R)$. 
\end{Lem}

\begin{proof}
 First we consider the case where $X$ is integral. We can restrict to the case where $f$ is surjective and $Y=\Spec A$, $X=\Spec B$ and $B$ is finitely generated as $A$-module. Let $K=\Quot(A)$ and $L=\Quot(B)$. Let $b \in B$ be an element whose minimal polynomial has coefficients in $A$ and degree $[L:K]$. Then the ring extension $A \subset A[b]$ is flat and $L=\Quot(A[b])$. Letting $X'=\Spec A[b]$ we find that $f$ factors as $f=g \circ h$ where $h: X \to X'$ is surjective and birational and $g: X' \to Y$ is finite and flat. The set \[S'=\{y \in Y(\R):\, 
 g^{-1}(y)\subset X'(\R)\}\smallsetminus g(\Supp \Omega_{X'/Y})\]is an open subset of $Y(\R)$ by Lemma \ref{lem:rfopenflat}. Clearly, $S$ is $S'$ minus a finite set of points and thus is also open.

 Now we consider the case where $X$ is irreducible. Let $X_{\rm red}$ be the reduced induced subscheme structure on $X$ and let $f': X_{\rm red} \to Y$ be the induced morphism. By the previous step, we have that the set \[S''=\{y \in Y(\R):\, 
 f'^{-1}(y)\subset X_{\rm red}(\R)\}\smallsetminus f'(\Supp \Omega_{X_{\rm red}/Y})\]is an open subset of $Y(\R)$. Since $S=S''\smallsetminus f(\Supp(\Omega_{X/Y}))$ and since the latter set is closed, $S$ is open.
 
 In the general case let $X_1,\ldots,X_r$ be the irreducible components of $X$ and let $f_i:X_i\to Y$ be the induced morphisms. Then $S$ is the intersection of the open sets \[\{y \in Y(\R):\, 
 f_i^{-1}(y)\subset X_{i}(\R)\}\smallsetminus f_i(\Supp \Omega_{X_{i}/Y})\]minus a finite set of points and therefore open.
\end{proof}

The previous lemma extends to the case of finite morphisms between arbitrary varieties by the curve selection lemma.

\begin{Lem}\label{lem:rfopen}
 Let $f: X \to Y$ be a finite morphism of varieties. The set \[S=\{y \in Y(\R):\, 
 f^{-1}(y)\subset X(\R)\}\smallsetminus f(\Supp \Omega_{X/Y})\]is an open subset of $Y(\R)$.
\end{Lem}

\begin{proof}
 Let $C \subset Y$ be an irreducible curve. By Lemma \ref{lem:curvesel2} it suffices to show that $C(\R)\cap S$ is open in $C(\R)$. Let $\pi: X\times_Y C\to C$ be the projection. Then $C(\R)\cap S$ equals the set \[\{y \in C(\R):\, 
 \pi^{-1}(y)\subset (X\times_Y C)(\R)\}\smallsetminus \pi(\Supp \Omega_{X\times_Y C/C})\] which is open.
\end{proof}

Now we are able to prove the main theorem of this section.

\begin{Thm} \label{thm:strictly_hyperbolic_open}
 Let $E \subset \pp^n$ be a linear subspace of dimension $n-k-1$. Let $H$ be the Hilbert scheme of closed subschemes of $\pp^n$ with a given Hilbert polynomial. Let $H' \subset H$ be the open subset consisting of all points corresponding to $k$-dimensional subschemesת that do not intersect $E$. The subset $S$ of $H'(\R)$ corresponding to subschemes which are hyperbolic with respect to $E$ such that the projection from $E$ is unramified at real points
 is open with respect to the classical topology.
\end{Thm}

\begin{proof}
 Let $W \subset \pp^n \times H'$ be a closed subscheme, flat over $H'$, all of whose fibers are closed subschemes of $\pp^n$ with Hilbert polynomial $P$ that do not intersect $E$. Let $\pi: \pp^n\smallsetminus E \to \pp^k$ be the linear projection from center $E$. The induced map $f=(\pi \times \textnormal{id})|_{W}: W \to \pp^k \times H'$ is quasi-finite and proper. Indeed, we have that $f$ composed with the projection $\pp^k\times H'\to H'$ is projective and then we can apply \cite[Cor. II-4.8(e)]{Hart77}. Thus it is finite by \cite[Thm. 8.11.1]{EGAIV3}. Therefore \[U=\{y \in \pp^k(\R) \times H'(\R):\, 
 f^{-1}(y)\subset W(\R)\}\smallsetminus f(\Supp \Omega_{W/\pp^k\times H'})\]is an open subset of $\pp^k(\R) \times H'(\R)$ by Lemma \ref{lem:rfopen}. A point $x \in H'(\R)$ is in $S$ if and only if the set $\pp^k(\R)\times\{x\}$ is fully contained in $U$. Thus, since $U$ is open and $\pp^k(\R)$ is compact, $S$ is open.
\end{proof}

\begin{Kor}\label{cor:smoothint}
  The set of equidimensional smooth hyperbolic subschemes is in the interior of the set of all hyperbolic subschemes.
\end{Kor}

\begin{proof}
 This follows from Theorem \ref{thm:strictly_hyperbolic_open} and \cite[Thm. 2.19]{us}.
\end{proof}

\begin{Bem}
 Let us stress another difference between the situation here and the one in the previous section. In the previous section when reducing to the case of curves smoothness was important. Since the curve selection lemma does not necessarily give us a smooth curve we had to use the normalization. This was not a problem because the normalization map is closed. But since it is not open this argument would not work in this section. However, we are lucky and we get along without smoothness of the target in this section.
 
 Let us further remark that if both $X$ and $Y$ are smooth of the same dimension in the crucial Lemma \ref{lem:rfopen}, then one can easily prove the statement using only the implicit function theorem:
 On the preimage of an open neighborhood of any point $y\in S$ the map $f$ is a covering map since it is unramified. The same is true if we restrict $f$ to the real points of $X$. Since the fiber over $y$ consists only of real points and the number of sheets is the same, for every point in a neighborhood of $y$ the fiber consists of real points only. Thus $S$ contains an open neighborhood of $y$. However, since the Hilbert scheme is not necessarily smooth, even at points that correspond to smooth varieties, this statement is not enough for our purposes.
\end{Bem}

The following two examples show that the set of smooth real subschemes of $\pp^n$ with the same Hilbert polynomial and hyperbolic with respect to $E$ can be disconnected and that one cannot always build even a topological isotopy between the real points of two hyperbolic curves in $\pp^3$.

\begin{example} \label{ex:different_writhe}
Consider the twisted cubic $(t^3: s t^2 : s^2 t: s^3)$ in $\pp^3$, this curve is hyperbolic with respect to the line spanned by $(4:0:1:0)$ and $(0:1:0:1)$ (see \cite[Ex.\ 4.15]{us} for details). Now consider another twisted cubic $(t^3 - s^2t: \frac{1}{2} s t^2 - s^2 t : \frac{1}{4} s t^2: 2 s t^2 -2 s^2 t - s^3)$, it is also hyperbolic with respect to the same line as the original one. However, by \cite{Bjo11} there exists no deformation passing through only smooth real curves that connects the two since the writhe of the first curve is $1$ and the writhe of the second is $-1$ (the second one is obtained from $(t^3: s t^2: s^2 t: - s^3)$ by applying a linear map with positive determinant).
\end{example}

\begin{example} \label{ex:shastri}
In \cite[\S3.1]{Sha92} Shastri constructed a representation of a long trefoil knot given by $(t^3 s^2 -3 t s^4 : t^4 s-4 t^2 s^3 : t^5 - 10 t s^4 : s^5)$. It turns out that this curve is hyperbolic with respect to the line spanned by $(0: 0: 1: -2)$ and $(1: -3: 21: -2)$. To see this note that the projection from this line gives rise to the map \[\pp^1\to\pp^1,\,(s:t)\mapsto(t^4s+3t^3s^2-4t^2s^3-9ts^4:2t^5-40t^3s^2+100ts^4+s^5).\] The B\'ezout matrix of these two polynomials is given by
\[
\begin{pmatrix}
9 & 4 & -3 & -1 & 0 \\
4 & 397 & 59 & -100 & -18 \\
-3 & 59 & 60 & -18 & -8 \\
-1 & -100 & -18 & 32 & 6 \\
0 & -18 & -8 & 6 & 2
\end{pmatrix}.
\]
It is easy to check that this matrix is positive definite. This means that the two polynomials have interlacing zeros and thus the map is real fibered, cf. \cite[Ex. 2.5]{us}. The same example from \cite{us} allows us to construct a degree $5$ unknot, for example if we consider the parametrization $[s^5: t s^4: (t+\frac{2}{3}s)(t+\frac{1}{3}s)(t-\frac{1}{3}s)(t-\frac{2}{3}s)(t-\frac{4}{3}s): (t+s)(t+\frac{1}{2}s)t(t-\frac{1}{2}s)(t-1)]$ we obtain a curve that is an unknot since its encomplexed writhe is $0$ (see \cite{Bjo11} for the classification of rational knots of low degree). The curve we obtain is hyperbolic with respect to the line $x_2 = x_3 = 0$ and of course we can rotate it to obtain an unknot hyperbolic with respect to the same line as the Shastri trefoil knot.

\end{example}

\section{Connectedness of the Hyperbolic Locus} \label{sec:connected}

In \cite{Har66} Hartshorne showed that the Hilbert scheme is connected. More modern treatment of this result as appears in \cite{Rev95} and \cite{Par96} uses the technique of Gr\"{o}bner degeneration. We, however, will use the original ideas of Hartshorne and thus some of his terminology. In particular, we will refer to a union of linear subspaces as ``fans'', we believe that no confusion with toric varieties will arise. 

In this section, we will check that the connecting path obtained by Hartshorne's construction passes only through schemes hyperbolic with respect to a fixed linear space. Let us denote by $E \subset \pp^n$ the space spanned by the last $n-k$ vectors of the standard basis, of dimension $n-k-1$ and let $E^{\perp}$ be the orthogonal complement of $E$ with respect to the standard inner product.

\begin{Lem} \label{lem:deform_to_perp}
Let $X$ be a real scheme hyperbolic with respect to $E$, then there exists a flat deformation of $X$ into a scheme with support $E^{\perp}$ over $\mathbb{A}^1$, such that the fiber over every $\R$-point is hyperbolic with respect to $E$.
\end{Lem}
\begin{proof}
For each $a\in\A^1\smallsetminus\{0\}$ consider the automorphism $\sigma_a$ of $\pp^n$ defined by $(x_0,\ldots,x_n)\mapsto(x_0,\ldots,x_k,ax_{k+1},\ldots,ax_{n})$. Then the $X_a=\sigma_a(X)$ form a flat family parametrized by $\A^1\smallsetminus\{0\}$ which extends uniquely to a flat family over $\A^1$. The fiber $X_0$ over $0$ agrees set theoretically with $E^\perp$ and for every $a\in\R$ the fiber $X_a$ is hyperbolic with respect to $E$.
%
\end{proof}

Recall that in \cite{Har66} Hartshorne defines for a coherent sheaf $\cF$ on a noetherian scheme $X$ the functors:
\[
R^i(\cF)(U) = \left\{ s \in \Gamma(U,\cF) \mid \codim \Supp(s) \geq i \right\}, \quad F^i(\cF) = \cF/R^i(\cF).
\]
Now let $\cF$ be a coherent sheaf on $\pp^n$. We denote by $n_i(\cF)$ the coefficient of $z^i$ in the Hilbert polynomial of $R^{n-i}(\cF)$ multiplied by $i!$. We write:
\[
n_*(\cF) = (n_k(\cF),\ldots,n_0(\cF)).
\]

\begin{Lem} \label{lem:n_*_increases}
Let $X$ be a real scheme hyperbolic with respect to $E$ and $Y$ be the scheme supported on $E^{\perp}$ obtained in the previous lemma, then $n_*(Y) \geq n_*(X)$ (with respect to pointwise ordering).
\end{Lem}
\begin{proof}
We apply \cite[Thm.\ 2.10]{Har66}, where $Y = \A^1$ and $F = \cO_X$.
\end{proof}


Let write $B_m \subset \GL_m(\R)$ for the subgroup of upper triangular matrices. The following lemma is a hyperbolic version of \cite[Cor.\ 5.3]{Har66}.

\begin{Lem} \label{lem:borel_invariant}
Let $X$ be a real closed subscheme of $\pp^n$ supported on $E^{\perp}$, then there exists a sequence of linear specializations $X = X_0, X_1, \ldots, X_r$ in $\pp^n$, such that each $X_j$ is supported on $E^{\perp}$ and $X_r$ is invariant under $B_{n+1}$. 
\end{Lem}
\begin{proof}
We can take a composition series for $B_{n+1}$, such that the intermediate quotients are either the additive or multiplicative group of $\R$ and that at each step we have semi-direct products. Now we would like to apply \cite[Prop.\ 5.2]{Har66}. We only need to verify that the construction in \cite[Prop.\ 5.2]{Har66} preserves hyperbolicity with respect to $E$. To see this we note that $E^{\perp}$, which is the support of $X$, is invariant (as a set and not pointwise) under the action of $B_{n+1}$, so when we perform the construction of \cite[Prop.\ 5.2]{Har66} the support of all the fibers is $E^{\perp}$, hence the scheme we obtain is hyperbolic with respect to $E$. 
%
\end{proof}

A monomial ideal $I \subset \R[x_1,\ldots,x_n]$ is an ideal that is an ideal generated by monomials. Equivalently, if $f \in I$, then every monomial in $f$ is also in $I$. The monomials form an ordered semi-group with respect to the coordinate-wise ordering. If $I$ is a monomial ideal, it is in fact generated by all the monomials in $I$ that are minimal with respect to this ordering. Let $f \in \R[x_1,\ldots,x_n]$ and let $ 1 \leq i < j \leq n$. We define the polynomial $f_{ij}$ to be the polynomial obtained from $f$ by replacing $x_i$ with $x_j$. A monomial ideal $I \subset \R[x_1,\ldots,x_n]$ is called Borel-fixed if whenever $f \in I$ and $1 \leq i < j \leq n$ are indices, we have that $f_{ij} \in I$.

\begin{Kor}
Let $X_r$ be the scheme obtained from $X$ in the previous lemma. Then $X_r$ is supported on $E^{\perp}$, hyperbolic with respect to $E$, it is cut out by a Borel-fixed ideal and satisfies $n_*(X_r) \geq n_*(X)$.
\end{Kor}
\begin{proof}
This follows immediately from the above lemma, \cite[Thm.\ 2.10]{Har66} and \cite[Prop.\ 5.4]{Har66}.
\end{proof}

Consider $ \R[t_{ij}]$, where $i$ and $j$ are positive integers. A canonical distraction of a monomial ideal $I \subset \R[x_1,\ldots,x_n]$ is the ideal generated by the expressions:
\[
\prod_{j=1}^{s_1} (x_1 - t_{1j} x_0) \cdots \prod_{j=1}^{s_n} (x_n - t_{nj} x_0),
\]
where $x_1^{s_1} \cdots x_n^{s_n}$ is a minimal (with respect to divisibility) monomial in $I$. 

A \textit{fan} is a subvariety of $\pp^n$ whose ideal is the intersection of prime ideals of the form $(x_{j_1} - a_1 x_0, x_{j_2} - a_2 x_0, \ldots, x_{j_r} - a_r x_0)$ where $1 \leq j_1 < j_2 < \cdots < j_r \leq n$. A \textit{tight fan} is a subvariety of $\pp^n$ whose ideal is the intersection of prime ideals of the form $(x_{j_1} , x_{j_2} , \ldots,x_{j_r-1}, x_{j_r} - a_r x_0)$ where $1 \leq j_1 < j_2 < \cdots < j_r \leq n$.

\begin{Lem} \label{lem:distractions}
Let $X$ be a real scheme supported on $E^{\perp}$ cut out by a Borel-fixed ideal. Applying distractions we obtain a fan $Y$ hyperbolic with respect to $E$ and $n_*(X) \leq n_*(Y)$.
\end{Lem}
\begin{proof}
We proceed along the lines of the proof of \cite[Thm.\ 4.10]{Har66}. We take the ideal of $X$ that is a monomial ideal generated  by monomials in $x_{k+1},\ldots,x_n$ (since it is supported on $E^{\perp}$). Let $A =\R[t_{ij}]_{\fm}$, where $i$ ranges between $k+1$ and $n$ and for each $i$, the index $j$ ranges between $1$ and the maximal power of $x_i$ in a monomial generating $I$ and $\fm$ is the irrelevant maximal ideal. Let $R = A[x_0,\ldots,x_n]$ and $J$ be the canonical distraction ideal of $I$ in $R$. Then $J$ cuts out a closed subscheme $X^{\prime} \subset \pp_A^n$ that is flat over $A$ and the closed point corresponds to $X$ and the generic point is a fan. This implies that for a generic choice of real numbers $t_{ij}$, the resulting specialization of $J$ is a fan in $\pp^n$ and does not intersect $E$ (since $E$ is cut out by $x_0 = \cdots = x_k = 0$ and in particular the only possible intersection is when all coordinates vanish). Since it is a fan and thus a union of linear subspaces that do not intersect $E$, we conclude that it is hyperbolic with respect to $E$. The last inequality follows from the observation after \cite[Thm.\ 4.10]{Har66}.
\end{proof}

Recall from \cite{Har66} that if $X$ is a fan in $\pp^n$ and we write the ideal of $X$ as an irredundant intersection $I = \bigcap_{j=1}^k \fp_j$, each prime is of the form $\fp_j = (x_{\ell} - a_{\ell,j} x_0, \ldots, x_{m_j} - a_{m,j} x_0)$, then we set $p = p(X)$ to be the largest integer, such that for every $j$ we have: 
\begin{itemize}
\item If $m_j < p$, then $a_{\ell,j} = \cdots = a_{m_j-1, j } = 0$.

\item If $m_j \geq p$, then $a_{\ell,j}= \cdots = a_{p-1,j} = 0$.

\end{itemize}
In particular, if all of the $a_{k,j} \neq 0$, then $p = 1$. Note that in the definition of \cite{Har66} $\ell = 1$. We need this minor modification since the fans we obtain from the construction in the previous lemma have $\ell = k+1$. Also note that for tight fans $p(X)$ is maximal.

We recall from \cite{Har66} the definition of linear specializations. Let $X$ be a scheme over a field $L$. We say that $x \in X$ specializes linearly to $x^{\prime}$ if there exists an extension $L_1/L$ and a morphism $\Spec  L_1[t]_{(t)} \to X$. Such that the generic point is mapped to $x$ and the closed point to $x^{\prime}$.

\begin{Lem} \label{lem:tight_fan}
Let $X \subset \pp^n$ be a fan hyperbolic with respect to $E$, and assume that the ideal of $X$ can be written as the following irredundant intersection:
\[
I = \bigcap_{j=1}^k (x_{k+1} - a_{d+1,j} x_0, \ldots, x_{m_j} - a_{m_j,j} x_0).
\]
Then there exists a chain of linear specializations from $X$ to $Y$, such that each is hyperbolic with respect to $E$ and the $Y$ is either a tight fan, $Y$ is a fan with $p(Y) > p(X)$ or $Y$ is a closed subscheme of $\pp^n$, satisfying $n_*(Y) > n_*(X)$.
\end{Lem}
\begin{proof}
We only need to verify that the transformations applied in the proof of \cite[Prop.\ 3.6]{Har66} preserve hyperbolicity with respect to $E$. Let $p=p(X)$. First note that from the proof of \cite[Prop,\ 3.6]{Har66} it follows that every change of coordinates of the form $x_p^{\prime} = x_p - \lambda x_0$, $x_{p+1}^{\prime} = x_{p+1} - \mu x_p$ and $x_j^{\prime} = x_j$ for $j \neq p, p+1$ results in a fan with the same $p$ and $n_*$. This is an automorphism of $\pp^n$ and preserves $E$ (since $p>k$), hence we can conclude that every fan in this family is hyperbolic with respect to $E$. Fix $\lambda,\mu\in\R$. Consider the fan in $\pp^n_{\R[t]}$ defined by
\[
\bigcap_{j=1}^k (x_{k+1}^{\prime} - b_{d+1,j} x_0^{\prime}, \ldots, x_{m_j}^{\prime} - b_{m_j,j} x_0^{\prime}).
\]
 Here the coefficients $b_{r,j}$ are $a_{r,j}$ if $r \neq p,p+1$ and $b_{p,j} = t (a_p-\lambda)$, $b_{p+1,j} = a_{p+1}-\mu a_p$. It is flat over $\Spec \R[t]$ and it does not intersect the subset of $\pp^n_{\R[t]}$ cut out by $x_0,\ldots,x_k$. Thus, the fiber over every $t\in\R$ is a scheme whose support is a fan that does not intersect $E$. For $t=1$ it is the fan that we obtained from $X$ by applying the linear transformation from the beginning of the proof and for a suitable choice of $\lambda$ and $\mu$ this limit is precisely the $Y$ we have been looking for.
\end{proof}

\begin{Thm} \label{thm:connected}
Let $H$ be the Hilbert scheme of closed subschemes of $\pp^n$ with given Hilbert polynomial. The set of schemes hyperbolic with respect to $E$ in $H(\R)$ is connected in the classical topology.
\end{Thm}
\begin{proof}
Let $X \subset \pp^n$ be a real subscheme with the prescribed Hilbert polynomial that is hyperbolic with respect to $E$. We will show that $X$ can be deformed into a tight fan hyperbolic with respect to $E$, such that the path passes only through schemes hyperbolic with respect to $E$. 

To do this we apply Lemma \ref{lem:deform_to_perp} to get a subscheme $X^{\prime}$ that is supported on $E^{\perp}$ and note that the Lemma guarantees that the path lies entirely in the closed set of points hyperbolic with respect to $E$ in $\cH(\R)$. Next we apply Lemma \ref{lem:borel_invariant} to obtain the subscheme $X^{\prime}_1$ that is monomial and still supported on $E^{\perp}$. Now we apply Lemma \ref{lem:distractions} to turn $X^{\prime}_1$ to a fan $X^{\prime}_2$ still hyperbolic with respect to $E$. By virtue of Lemma \ref{lem:tight_fan} we can deform $X^{\prime}_2$ into $X^{\prime}_3$ that is hyperbolic with respect to $E$. If $X^{\prime}_3$ is a tight fan we are done. If $X^{\prime}_3$ is a fan again we can repeat the last step, else $X^{\prime}_3$ is a close subscheme of $\pp^n$, such that $n_*(X_3^{\prime}) > n_*(X^{\prime}_2) \geq n_*(X)$ and thus we can repeat the entire process starting with $X_3^{\prime}$. Since the invariant $n_*$ only increases and since by \cite[Cor. 3.10]{Har66} there are only finitely many possibilities for $n_*(Z)$ where $Z$ is a fan, we will end up with a tight fan after a finite number of steps.

Now the claim follows from \cite[Prop. 3.2]{Har66} and its proof which shows that for any two tight fans $X_1$ and $X_2$ with the same Hilbert polynomial there is a flat family of tight fans over some $\A^m$ which has $X_1$ and $X_2$ as fibers at some closed points.
\end{proof}

\section{Smoothing nodes} \label{sec:smoothing}

\begin{Def}
 A hyperbolic subscheme $X\subset\pp^n$ is said to be \textit{hyperbolically smoothable} if its corresponding point is in the closure (with respect to the classical topology) of the subset of all points corresponding to hyperbolic subschemes without real singularities.
\end{Def}

A result by Nuij \cite{nuij} says that every hyperbolic hypersurface is hyperbolically smoothable. In general, we cannot expect that every hyperbolic scheme is hyperbolically smoothable since it might not be contained in a connected component of the Hilbert scheme with points corresponding to nonsingular varieties. This is illustrated in the next example. However, if the variety is Cohen--Macaulay and if it has only ordinary double points for singularities, then we are able to give a tractable criterion for hyperbolic smoothability.

\begin{example} \label{ex:nonsmoothable}
 Let $C\subset \pp^3$ be a hyperbolic plane quartic curve and let $L\subset \pp^3$ be a line that intersects $C$ in one point. Then the curve $C\cup L$ is hyperbolic and it is not smoothable since its arithmetic genus is three and its degree is five, cf. \cite[4.3.2]{HarHir}.
\end{example}

Let $B$ be a finitely generated $\R$-algebra, $I\subset B$ be an ideal such that $A=B/I$ is a finite dimensional $\R$-vector space and such that $\Spec A$ consists only of $\R$-points. Let $D=\R[\epsilon]/(\epsilon^2)$ be the ring of dual numbers, $\pi: D\to \R$ the natural projection and $B'=B\otimes_\R D$. Consider $\varphi\in\Hom_B(I,B/I)$ and let \[ I'=\{x+\epsilon y : \, x\in I, y\in B, \textnormal{ and the image of }y\textnormal{ in }B/I\textnormal{ is equal to }\varphi(x)\}.\] Then, $A'=B'/I'$ is flat over $D$ and we have the natural projection $\pi':A'\to A$. The trace maps commute with the projection maps, i.e. $\pi\circ\tr_{A'/D}=\tr_{A/\R}\circ\pi'$. Thus, we have that $\tr_{A'/D}(\pi'^{-1}(\Nil(A)))\subset (\epsilon)\subset D$. We identify $(\epsilon) \cong \R$ via $\epsilon\mapsto 1$. This is an isomorphism of $D$-modules. We define the $\R$-bilinear form $b_\varphi$ on $\Nil(A)$ as follows: For $f,g\in \Nil (A)$ let $b_\varphi(f,g)=\tr_{A'/D}(f' g')\in (\epsilon)\cong \R$ where $f'\in\pi'^{-1}(
f)$ and $g'\in\pi'^{-1}(g)$. Note that the value of $\tr_{A'/D}(f' g')$ depends only on the choice of $f$ and $g$ rather than on the particular choice of the preimages $f'$ and $g'$.

\begin{Def}
 We say that $\varphi\in\Hom_B(I,B/I)$ is a \textit{strict hyperbolic deformation} if $b_\varphi$ is positive definite.
\end{Def}

\begin{Bem}
 If $A$ is reduced, then every element of $\varphi\in\Hom_B(I,B/I)$ is a strict hyperbolic deformation.
\end{Bem}

\begin{Bem}\label{rem:socle}
 In the situation above let $A=A_0\times\cdots\times A_r$ where the $A_i$ are local $\R$-algebras. Let $p_i\in\Spec A$ be the point corresponding to the maximal ideal of $A_i$. Assume that $A_i$ is reduced for $1\leq i\leq r$ and that the vector space dimension of $A_0$ is two. Let $f \in \sqrt{I}\smallsetminus I$. Note that $f^2\in I$. Let $\varphi\in\Hom_B(I,B/I)$. If $\varphi(f^2)(p_0)<0$, then $\varphi$ is a strict hyperbolic deformation. Indeed, let $g\in A'$ be the residue class of $f$. Then we have $\tr_{A'/D}(g^2)=-\varphi(f^2)(p_i)\epsilon$. Thus, $b_\varphi$ is positive definite.
\end{Bem}

\begin{example}
 Let $B=\R[x,y]$ and $I=(g_1, g_2)$ with $g_1=1-x^2-y^2$ and $g_2=x^2-x$. The zero dimensional scheme $\Spec A$ where $A=B/I$ consists of the two reduced points at $(0,\pm 1)$ and one double point at $(1,0)$. In order to apply the previous remark we consider $xy \in \sqrt{I}\smallsetminus I$. We have \[x^2y^2=-xg_1+(y^2-x-1)g_2.\] Therefore, the strict hyperbolic directions are precisely those $\varphi\in\Hom_B(I,B/I)$ with $\varphi(g_1+2g_2)(1,0)>0$.
\end{example}

\begin{example}
 Let $B=\R[x]$ and $I=(x^3)$. Every flat deformation of $A=B/I$ in $B$ over $D$ is of the form $B'/I'$ where $B'=B\otimes_\R D$ and $I'=(x^3-\epsilon\cdot(ax^2+bx+c))$ for some $a,b,c\in\R$. It is a strict hyperbolic deformation if the bilinear form defined above is positive definite. A representing matrix of this bilinear form is
 \[
  \begin{pmatrix}
   2b&3c\\3c&0
  \end{pmatrix}.
 \]This matrix is not positive definite for any value of $b,c\in\R$. Thus, there are no strict hyperbolic deformations of $A$ in $B$ over $D$.
\end{example}

%

Now let $X\subset\pp^n$ be a closed Cohen--Macaulay subscheme of pure dimension $k$ which is hyperbolic with respect to a linear subspace $E\subset\pp^n$ of dimension $n-k-1$. Let $E'\subset\pp^n$ be a real linear subspace of dimension $n-k$ which contains $E$ and let $X'=X\cap E'$ be the scheme theoretic intersection. Since $X$ is Cohen--Macaulay, we get a natural homomorphism \[H^0(X,\cN_{X/\pp^n})\to H^0(X',\cN_{X'/E'}).\] Since $X'$ is zero dimensional, we can make the following definition.

\begin{Def}
 Let $X\subset\pp^n$ be a closed Cohen--Macaulay subscheme of pure dimension $k$ which is hyperbolic with respect to a linear subspace $E\subset\pp^n$ of dimension $n-k-1$. A global section of the normal sheaf $\cN_{X/\pp^n}$ is a \textit{strict hyperbolic deformation} if for every real linear subspace $E'\subset\pp^n$ of dimension $n-k$ which contains $E$ its image in $H^0(X',\cN_{X'/E'})$ with $X'=X\cap E'$ is a strict hyperbolic deformation.
\end{Def}

\begin{Lem}
 Let $T$ be an integral noetherian scheme, $d\in\Z$ and $X\subset\pp^n_T$ a closed subscheme such that for every $t\in T$ we have that $X_t=X\times_T\Spec \kappa(t)\subset\pp^n_{\kappa(t)}$ is Cohen--Macaulay and has degree $d$. Then every finite surjective linear projection $\pi:X\to\pp^k_T$ is flat.
\end{Lem}

\begin{proof}
 It suffices to show that for every $p\in\pp^k_T$ we have that the Hilbert polynomial of $X\times_{\pp^k_T}\Spec\kappa(p)$ is constant $d$ \cite[Thm. III-9.9]{Hart77}. Let $t\in T$ be the image of $p$ under the projection $\pp^k_T\to T$. We have the following commuting diagram of morphisms:
  \[
  \begin{tikzcd}
 (X_t)_{\kappa(p)}= X_t \times_{\kappa(t)} \Spec\kappa(p) \arrow{d}\arrow{r} & \pp^k_{\kappa(p)} \arrow{d}\\
  X_t \arrow{r} \arrow{d} & \pp^k_{\kappa(t)} \arrow{d}\\
   X \arrow{r} & \pp^k_T
   \end{tikzcd}.
 \]Since $X_t$ is Cohen--Macaulay of degree $d$, the linear projection $X_t\to\pp^k_{\kappa(t)}$ is flat of degree $d$. Thus, the same is true for the base change $(X_t)_{\kappa(p)}\to\pp^k_{\kappa(p)}$ and all of its fibers have Hilbert polynomial $d$.
\end{proof}

\begin{Kor}
 Let $A$ be a discrete valuation ring and $X\subset\pp^n_A$ be a closed subscheme which is flat over $A$. If the fiber $X_0$ over the closed point of $A$ is Cohen--Macaulay, then every finite surjective linear projection $\pi:X\to\pp^k_A$ is flat.
\end{Kor}

\begin{proof}
 Since $X$ is flat over $A$, the fiber $X_1$ over the generic point has the same degree as $X_0$. Since having a Cohen--Macaulay fiber is an open condition \cite[Thm. 12.2.1(vii)]{EGAIV3}, $X_1$ is also Cohen--Macaulay. Thus, we can apply the previous lemma.
\end{proof}

\begin{Thm}\label{thm:smoothable}
 Let $R=\R\{\{\epsilon\}\}$ be the field of Puiseux series over $\R$. Let $X\subset\pp^n$ be a closed Cohen--Macaulay subscheme of pure dimension $k$ which is hyperbolic with respect to the linear subspace $E\subset\pp^n$. Let $X'\subset \pp^n_{ \R[[\epsilon]]}$ be a flat deformation of $X$ over $\R[[\epsilon]]$. If the induced flat deformation of $X$ over $D=\R[[\epsilon]]/(\epsilon^2)$ is a strict hyperbolic deformation, then $X'_R=X'\times_{\R[[\epsilon]]} \Spec R$ is a hyperbolic subscheme of $\pp^n_R$ without singular $R$-points.
\end{Thm}

\begin{proof}
  Let $\pi: X'\to \pp^k_{\R[[\epsilon]]}$ be the linear projection from $E$. It follows from our assumptions that $\pi$ is finite and flat by the preceding corollary. Let $v:R\to\Q\cup\{\infty\}$ be the natural valuation on the field of Puiseux series. Letting $p=(p_0:\ldots:p_k)\in\pp^k_R$ with $p_i\in R$ not all zero, we have to show that the trace bilinear form that we get from the projection $X'_R\to\pp^k_R$ is positive definite at $p$. Without loss of generality we can assume that $v(p_0)\leq v(p_i)$ for all $1\leq i\leq k$. In particular, $p_0\neq 0$ and $p=(1:q_1:\ldots:q_k)$ where $v(q_i)\geq0$ for $1\leq i\leq k$. 
  Let $U\subset\pp^k_{\R[[\epsilon]]}$ be the open affine subset given by $x_0\neq0$ and let $\pi^{-1}(U)=\Spec A_1$. We have $U=\Spec A_0$ with $A_0=\R[[\epsilon]][x_1,\ldots,x_k]$. As an $A_0$-module $A_1$ is finitely generated and flat, thus it is projective. By the Quillen--Suslin theorem (cf. e.g. \cite[Thm. 8.5]{brunsktheory}) $A_1$ is actually free as $A_0$-module. Thus, after choosing a basis we can represent the trace bilinear form of the $A_0$-algebra $A_1$ with a symmetric matrix $H$ with entries in $A_0$. We can write \[H=H_0+\epsilon H_1 + \epsilon^2 H_2,\] where $H_0, H_1$ are symmetric matrices of some size $d$ with entries in $\R[x_1,\ldots,x_k]$ and $H_2$ is a symmetric matrix with entries in $A_0$ of the same size. The trace bilinear form of the projection $X'_R\to\pp^k_R$ at $p$ is represented by the matrix $H(q)$ that we get from $H$ by substituting $x_i$ by $q_i$. Also note that the trace bilinear form of the projection $X\to\pp^k_\R$ from $E$ on the open affine subset $x_0\neq0$ is 
represented by $H_0$. Let $0\neq w\in R^d$ be some vector. We have to show that $w^T H(q) w>0$. Without loss of generality we can assume that $v(w_j)=0$ for some $1\leq j\leq d$ and $v(w_i)\geq0$ for all $1\leq i\leq d$. Since $X$ is hyperbolic, $H_0$ is positive semidefinite. Thus, the symmetric matrix $H_0(q)$ is positive semidefinite and in particular, if $v(w^TH_0(q)w)=0$, then $w^T H(q) w>0$. Otherwise we still have $w^T H_0(q) w\geq0$ and since the induced deformation over $D$ is a strict hyperbolic deformation, we have $v(w^TH_1w)=0$ and $w^T H_1(q) w>0$. This implies $w^T H(q) w>0$.  
\end{proof}

\begin{Kor}
 Let $X\subset\pp^n$ be a closed Cohen--Macaulay subscheme of pure dimension $k$ which is hyperbolic with respect to the linear subspace $E\subset\pp^n$. Assume that the Hilbert scheme is nonsingular at the point $x$ corresponding to $X$. If there exists a strict hyperbolic deformation in $H^0(X,\cN_{X/\pp^n})$, then $X$ is hyperbolically smoothable.
\end{Kor}

\begin{Bem}
 Let $X\subset\pp^n$ be a hyperbolic subscheme, let $T$ be a smooth, irreducible curve over $\R$ and let $X'\subset\pp^n\times T$ be a subscheme, flat over $T$ such that the fiber over $t_0\in T(\R)$ is $X$. Proposition \ref{thm:smoothable} gives a criterion on the induced flat deformation over the dual numbers to check whether $t_0$ is in the closure of the set of points $t\in T(\R)$ whose fiber is a smooth hyperbolic subscheme. In general, for deciding this question it is not enough to look at the induced flat deformation over the dual numbers as the following example shows: Let $B=\R[x]$ and $I=(x^2)$. The following two flat deformations of $A=B/I$ in $B$ over $\R[t]$ given by $I_1=(x^2+t\cdot x)\subset B[t]$ and $I_2=(x^2+t\cdot x+t^2)\subset B[t]$ give rise to the same deformation over $D$. But while in the first case every fiber over $t\neq0$ consists of two reduced real points, we find that in the second case no fiber consists of two reduced real points.
\end{Bem}

\section{Reciprocal linear spaces}

Recall, for example from \cite{SSV13}, that a \textit{reciprocal linear space} is the Zariski closure of the image of a linear space under the Cremona transform. These projective varieties are hyperbolic \cite{Var95}, but usually very singular \cite{SSV13}. In the following two examples we will use the methods developed in the previous section to get new examples of \textit{smooth} hyperbolic varieties from reciprocal linear spaces. After that, we will show that the reciprocal of a general two-dimensional linear space is not hyperbolically smoothable given that its codimension is at least two.

\begin{example}
 We consider the reciprocal linear space $X \subset \mathbb{P}^4$ of dimension two and degree four from \cite[Example 1]{SSV13}. We have that $X$ is the common zero set of the two polynomials $f=x_0 x_1-x_0 x_3 - x_1 x_3$ and $g=x_0 x_2-x_0 x_4-x_2 x_4$. It is hyperbolic with respect to the line $E$ spanned by $(1:1:0:-1:0)$ and $(1:0:1:0:-1)$. It has four singularities, all of them real and at infinity, namely $p_1=(0:1:0:0:0)$, $p_2=(0:0:1:0:0)$, $p_3=(0:0:0:1:0)$ and $p_4=(0:0:0:0:1)$. Let $S=\mathbb{R}[x_0,\ldots,x_4]$ and $I=(f,g) \subset S$. Using Remark \ref{rem:socle} we find that $\varphi\in\textrm{Hom}_S(I,S/I)$ is a strict hyperbolic deformation if $\varphi(g_2)(p_1)<0$, $\varphi(g_1)(p_2)<0$, $\varphi(g_2)(p_3)<0$, $\varphi(g_1)(p_4)<0$. Let $q=(x_1+x_3)^2+(x_2+x_4)^2$. The variety $X_t \subset \mathbb{P}^4$ cut out by the two polynomials $f_t=f-t q$ and $g_t=g-t q$ is smooth and disjoint from $E$ for all $t>0$. Thus, by the Proposition 
 \ref{thm:smoothable} and Corollary \ref{cor:smoothint} $X_t$ is hyperbolic for all $t>0$.
\end{example}

\begin{example}
 Let $\cL^{-1}\subset \pp^5$ be the reciprocal of the row span $\cL$ of the matrix
\[
 \begin{pmatrix}
  1&0&0&0&1&1\\
  0&1&0&0&1&1\\
  0&0&1&0&1&0\\
  0&0&0&1&0&1
 \end{pmatrix}.
\]The threefold $\cL^{-1}$ has degree $7$ and it is cut out by the three cubic forms
\[y_0y_1y_2-y_0y_1y_4-y_0y_2y_4-y_1y_2y_4,\]
\[y_0y_1y_3-y_0y_1y_5-y_0y_3y_5-y_1y_3y_5,\]
\[y_2y_3y_4-y_2y_3y_5-y_2y_4y_5+y_3y_4y_5.\]
The singular locus of $\cL^{-1}$ consists of all lines spanned by two standard unit vectors except for the three lines $\overline{e_0 e_1}$, $\overline{e_2 e_4}$ and $\overline{e_3 e_5}$. The linear projection from $\cL^{\perp}$ is given by
\[\pi:\cL^{-1}\to\pp^3, \, y\mapsto(y_0+y_4+y_5:y_1+y_4+y_5:y_2+y_4:y_3+y_5).\]
We consider the intersection $X$ of $\cL^{-1}$ with the preimage of the hyperplane $H=\cV_+(x_0-x_1+x_2+x_3)$ under $\pi$. Then we get, in new coordinates $z_i$, the variety $X\subset\pp^4$ cut out by the ideal $I\subset S=\R[z_0,z_1,z_2,z_3,z_4]$ generated by the following three cubic forms\[f_1=z_1 z_2 z_3 - z_1 z_2 z_4 - z_1 z_3 z_4 + z_2 z_3 z_4,\]
\[f_2=z_0^2 z_2 + z_0 z_1 z_2 + z_0 z_2^2 + z_0 z_2 z_3 - z_0^2 z_4 - z_0 z_1 z_4 -  2 z_0 z_2 z_4 - z_1 z_2 z_4 - z_2^2 z_4 - z_0 z_3 z_4 \]\[- z_2 z_3 z_4 - z_0 z_4^2 -   z_2 z_4^2,\]
\[f_3=z_0^2 z_1 + z_0 z_1^2 + z_0 z_1 z_2 - z_0^2 z_3 - 2 z_0 z_1 z_3 -   z_1^2 z_3 - z_0 z_2 z_3 - z_0 z_3^2 - z_1 z_3^2 + z_0 z_1 z_4 \]\[- z_1 z_2 z_4 -   z_0 z_3 z_4 - 2 z_1 z_3 z_4 + z_2 z_3 z_4.\]
The restriction of $\pi$ to $X$ is given by the linear projection
\[X\to\pp^2, \, z\mapsto (z_0+z_3+z_4:z_1+z_3:z_2+z_4),\]i.e. $X$ is hyperbolic with respect to $E=\cV_+(z_0+z_3+z_4,z_1+z_3,z_2+z_4)$. The singular locus of $X$ consists of twelve nodes. Let
\[g_1=(-z_2+z_4)(z_2^2+z_4^2)+(z_1-z_3)(z_1^2+z_3^2),\]
\[g_2=(2z_0+z_1+3z_2+z_3-z_4)(z_1^2+z_3^2)+(-z_2+z_4)(z_2^2+z_4^2),\]
\[g_3=3(z_1-z_3)(z_1^2+z_3^2)+2(z_0+z_3+z_4)(z_2^2+z_4^2).\]
 Using a computer algebra system one checks that $X$ is locally a complete intersection and that there is a $\varphi\in\Hom_S(I,S/I)$ satisfying $\varphi(f_i)=g_i$.
Thus \[X_t=\{(p,t)\in\pp^4\times\A^1:\, f_i(p)=t\cdot g_i(p),\,\, i=1,2,3 \}\] is a flat deformation of $X$ in $\pp^4$ over $\A^1$. The fiber over $t\in\A^1$ is nonsingular and does not intersect $E$ for every $0<t<0.08$. Using Remark \ref{rem:socle} and a computer algebra system it is not hard to show that the induced deformation over the dual numbers is a strict hyperbolic deformation. Thus, by Proposition \ref{thm:smoothable} and Corollary \ref{cor:smoothint} the fiber over $t\in\A^1$ is a smooth hyperbolic variety whenever $0<t<0,08$.
\end{example}

The idea to the following arose in a discussion with Mateusz Micha\l{}ek and Kristin Shaw. Let $A\in \R^{3\times n}$ with $n\geq5$ be a real matrix none of whose maximal minor vanishes. Let $\cL\subset\pp^{n-1}$ be the plane spanned by the rows of $A$ and $\cL^{-1}$ its reciprocal, i.e. the Zariski closure of the image of $\cL$ under coordinatewise inversion. We will show that $\cL^{-1}$ is not hyperbolically smoothable. In $\cL$ we have the lines $l_1,\ldots,l_n$ that are the intersections of $\cL$ with the coordinate hyperplanes. The variety $\cL^{-1}$ is the blow-up of $\cL$ at the $\binom{n}{2}$ points where two of the $l_i$ intersect. Its singularities are the points $p_i$ which are the images of the $l_i$. The line $g_{ij}$ which is spanned by $p_i$ and $p_j$ for $i\neq j$ is the exceptional divisor over the point $q_{ij}=l_i\cap l_j$. The projection $\pi:\cL^{-1}\to\pp^2$ from $\cL^{\perp}$ is real fibered \cite{Var95}. Let $H\subset\pp^2$ be a line not intersecting the projections of the $p_i$ and whose preimage $C=\pi^{-1}(H)\subset\cL^{-1}$ is smooth. Let $C'\subset\cL$ be the Zariski closure of the image of $C$ under the Cremona transform. 

\begin{Lem}
 The curve $C'$ intersects each $l_i$ transversally in the $n-1$ points $q_{ij}$ and in no more points. In particular, the degree of $C'$ is $n-1$.
\end{Lem}

\begin{proof}
 The line $H$ intersects $\pi(g_{ij})$ transversally in exactly one point. Thus, $C$ intersects each $g_{ij}$ transversally in exactly one point. Since $C$ does not contain any of the $p_i$, the claim follows.
\end{proof}

 In particular, the lemma implies that the curve $C'$ is smooth and thus isomorphic to $C$. Furthermore, as shown in \cite{CentralCurve}, the curve $C'\subset\cL$ is hyperbolic as it is the derivative of a product of linear forms. Thus, the real part $C(\R)$ has $\lceil \frac{n-1}{2}\rceil$ connected components.
 
 \begin{Prop}
  For $n\geq5$ the reciprocal $\cL^{-1}$ of a generic two dimensional linear space is not hyperbolically smoothable.
 \end{Prop}

\begin{proof}
    Let $T$ be a smooth curve and consider a closed subscheme $X \subset \pp^{n-1}_T$ such that the projection $X \to T$ is flat. Let $t_0\in T(\R)$ such that $\cL^{-1} = X_{t_0}$ and assume that every neighbourhood of $t_0 \in T (\R)$ contains a $t$ with $X_t$ smooth and hyperbolic with respect to $L^{\perp}$. Without loss of generality we can assume that for all $t \in T (\C)$ we have that $X_t \cap L^{\perp} = \emptyset$. We can consider the projection $\pi : X \to \pp^2_T$ from center $\cL^{\perp}$. Let $Y \subset \pp_T^{n-1}$ be the preimage of the line $H$ under this projection. This gives rise to the flat family $Y \to T$ with $C = Y_{t_0}$. Since $C$ is smooth, the fibers $Y_t$ for $t$ in a sufficiently small neighbourhood $U$ of $t_0$ have real parts with  $\lceil \frac{n-1}{2}\rceil$ connected components. On the other hand, the degree of $\cL^{-1}$ is $d=\binom{n-1}{2}$ \cite{PS06} meaning that $X_t(\R)$ is a $d$-sheeted covering space of $\R\pp^2$ for $t \in U$. This implies that $Y_t(\R)$ for $t\in U$ has at least $\lceil\frac{d}{2}\rceil$ connected components. Thus, we have $\lceil\frac{(n-1)(n-2)}{4}\rceil\leq \lceil\frac{n-1}{2}\rceil$. For $n\geq5$ this is impossible.
\end{proof}

\bigskip


 \bibliographystyle{alpha}
 \bibliography{biblio}
 \end{document}